\documentclass[11pt]{amsart}
\usepackage{amsfonts}
\usepackage{mathrsfs}
\usepackage[utf8]{inputenc} 

\usepackage[T1]{fontenc}    
\usepackage{bbm}%
\usepackage{bm}
\usepackage{bbding}
\usepackage[mathcal]{euscript}
\usepackage{graphicx}
\usepackage{amsthm,amscd}
\usepackage{calc}
\usepackage{mathtools}
\usepackage{mathrsfs,dsfont}
\usepackage{CJK,fancyhdr,amscd}
\usepackage[dvipsnames]{xcolor}
\usepackage{anysize}
\usepackage{amsmath,amssymb,amsfonts}
\usepackage{epsfig,enumerate}
\usepackage{latexsym}
\usepackage{indentfirst, latexsym}
\usepackage{graphics}
\usepackage[all,poly,knot]{xy}
\usepackage[pagebackref]{hyperref}
\usepackage{xpatch}
\usepackage{hyperref}
\usepackage[T1]{fontenc}
\usepackage{textcomp}
\usepackage[utf8]{inputenc}
\usepackage[right]{lineno}

\hypersetup{
	colorlinks=true,
	linkcolor=Blue,
	filecolor=YellowOrange,
	citecolor = WildStrawberry,
	urlcolor=cyan,
}

\xpatchcmd{\step}{%
	\normalfont\scshape\centering}{%
	\normalfont\scshape}{\typeout{Success}}{\typeout{Failure}}%

\allowdisplaybreaks[2]

\providecommand{\U}[1]{\protect\rule{.1in}{.1in}}

\numberwithin{equation}{section}

\def\cX{\mathcal{X}}

\def\cal C{\mathcal{C}}

\def\inf{\rm inf}

\theoremstyle{plain}
\newtheorem{theo}{Theorem}[section]
\newtheorem{lemma}[theo]{Lemma}
\newtheorem{main theorem}[theo]{Main theorem}

\newtheorem{prop}[theo]{Proposition}
\newtheorem{cor}[theo]{Corollary}
\newtheorem{claim}[theo]{Claim}

\newtheorem{conj}[theo]{Conjecture}

\theoremstyle{definition}
\newtheorem{defn}[theo]{Definition}

\newtheorem{exam}[theo]{Example}
\newtheorem{step}[]{Step}

\newcommand{\comment}[1]{}

\usepackage{fancyhdr}
\begin{document}

\title{Rigidity for compact hyperbolic complex manifolds}

\author{Mu-lin Li}
\address{School of Mathematics, Hunan University, China}
\email{mulin@hnu.edu.cn}

\author{Sheng Rao}
\address{School of Mathematics and statistics, Wuhan  University, Wuhan 430072, China}
\email{likeanyone@whu.edu.cn}

\author{Mengjiao Wang}
\address{School of Mathematics and statistics, Wuhan  University, Wuhan 430072, China}
\email{mjwang96@whu.edu.cn}
\thanks{The authors are are partially supported by NSFC (Grant No. 12271412, W2441003) and Hubei Provincial Innovation Research Group Project (Grant No. 2025AFA044).}
\date{\today}
\subjclass[2020]{Primary 32Q45; Secondary 32G05,  14D22, 53C24, 32G13}
\keywords{Hyperbolic and Kobayashi hyperbolic manifolds; Deformations of complex structures, Fine and coarse moduli spaces, Rigidity results, Complex-analytic moduli problems}

\begin{abstract}
We study the deformation behavior of compact hyperbolic complex manifolds. Let $\pi:\mathcal{X}\rightarrow \Delta$ be a smooth family of compact complex manifolds over the unit disk in $\mathbb{C}$, and $H$ a compact hyperbolic complex manifold. Then the $H$-locus 
$\{t\in\Delta: X_t\cong H\}$
is either at most a discrete subset of $\Delta$ or the whole $\Delta$. For a smooth family over a compact Riemann surface $Y$, its $H$-locus 
is either at most finite or the whole $Y$. Furthermore, if $Y$ is isomorphic to $\mathbb{P}^1$ or an elliptic curve, then we conjecture that the $H$-locus is empty or the whole $Y$.  
\end{abstract}

\maketitle

\section{introduction}
A \emph{smooth family} $\pi:\mathcal{X}\rightarrow\Delta$ of compact complex manifolds is defined as a proper holomorphic submersion, such that all fibers $X_{t}:=\pi^{-1}(t)$ for $t\in\Delta$ are compact complex manifolds. The deformations of Kobayashi hyperbolic manifolds have been studied in many papers, such as \cite{Ko73,gf75,Ka76,Wr77,Br78,Za88}. Throughout this paper, the term ``hyperbolic'' specifically refers to ``Kobayashi hyperbolic''.

Two fundamental questions arise concerning the deformation of such manifolds:
\begin{enumerate}
    \item Is the subset of {hyperbolic complex  structures}  \textit{closed} in the full space of all complex structures?
    \item Is the subset of {hyperbolic complex  structures} \textit{open} in the full space of all complex structures?
\end{enumerate}
The answer to the first question is negative in general, while the second is affirmative.
In the following, we outline the historical developments and key results regarding the deformations of hyperbolic complex structures.

R. Brody--M. Green \cite{BG77} constructed a smooth family $\{X_t\}$ of high-degree hypersurfaces in $\mathbb{P}^3$ parameterized by a disk $\Delta$, where all fibers $X_t$ for $t \neq 0$ are hyperbolic, while the central fiber $X_0$ is not hyperbolic. Subsequently, M. Wright {\cite[COROLLARY 4.1]{Wr77}} established that for any smooth family $\pi: \mathcal{X} \to \Delta$ of compact complex manifolds, if $X_{t_i} \cong H$ for some fixed hyperbolic manifold $H$ and a sequence $\{t_i\}_{i=1}^{+\infty} \subseteq \Delta^* := \Delta \setminus \{0\}$ converging to $0$, then the central fiber $X_0$ must also be hyperbolic. In most places of this paper, ``$\cong$'' denotes the biholomorphism between two compact complex manifolds.
These results collectively resolve the first question regarding the closedness of the coarse moduli space in the analytic topology.

The following address the second one. Brody \cite[THEOREM 3.1]{Br78} proved that the points in $\Delta$ corresponding to hyperbolic complex structures in compact complex fibers $X_{t}$ form an open Euclidean set of $\Delta$.  Furthermore, S. Kobayashi \cite[(3.11.1) Theorem]{Ko98} established that for a complex fiber space $(\mathcal{X},f,B)$, if there is a point $b_{0}\in B$ such that the fiber $X_{b_0}$ is hyperbolic, then there is an open Euclidean neighborhood $U_{b_0}$ of $b_0$ such that $X_{b}$ is hyperbolic for any $b\in U_{b_0}$.
Recall that a \emph{complex fiber space} consists of complex spaces $\mathcal{X},\ B$ and a surjective holomorphic map $f:\mathcal{X}\rightarrow B$ with compact fibers.

Based on these works and moduli theory of hyperbolic structures, and much inspired by the uncountability argument in \cite{RT21,RT22}, we prove a smooth deformation rigidity result of hyperbolic structures:
\begin{theo}\label{main theorem}
Let $\pi: \cX\rightarrow\Delta$ be a smooth family of compact complex manifolds, and $H$ a compact hyperbolic complex manifold. Then the $H$-locus 
$$\{t\in\Delta: X_t\cong H\}$$
is either at most a discrete subset of $\Delta$ or the whole $\Delta$.
\end{theo}

Our proof strategy of Theorem \ref{main theorem}\footnote{\label{fn1} If Kobayashi's conjecture holds true, then Theorem \ref{main theorem} can be also obtained by the moduli theory of canonically polarized projective manifolds. This is pointed out to us by Professor Kang Zuo.} hinges on two fundamental properties: the openness of hyperbolicity in deformation theory and the Hausdorff property of the coarse moduli space $\mathfrak{M}$ given by Wright in \cite[THEOREM 7]{Wr77} (a detailed proof given in Proposition \ref{sep} here). This moduli space parametrizes isomorphism classes of compact hyperbolic complex manifolds.  Suppose that there is a subset $E\subset \Delta$ with accumulation points in $\Delta$ such that $X_{t}\cong H$ for any $t\in E$. The crux of our argument lies in finding an open subset $U\subset\Delta$ such that:
\begin{enumerate}[(1)]
\item All fibers $X_{t}$ for $t\in U$ are hyperbolic;
\item The open subset $U$ contains an accumulation point of $E$ in $\Delta$.
\end{enumerate}
For the coarse moduli space $\mathfrak{M}$ of compact hyperbolic complex manifolds, there exists a holomorphic morphism $\sigma_{U}:U\rightarrow \mathfrak{M}$ sending each point $t\in U$ to the corresponding isomorphism class $[X_t]$ of the hyperbolic manifold
$X_{t}$ in $\mathfrak{M}$. The image $\sigma_{U}(U\cap E)$ consists of a single point $[H]$. By the Hausdorff property of $\mathfrak{M}$, $[H]$ is closed in
$\mathfrak{M}$. Since $\sigma_U$ is holomorphic and its preimage $\sigma_{U}^{-1}([H])$ is not discrete, it follows that $\sigma_{U}^{-1}([H])$ must be the whole set $U$. Thus, $X_{t}\cong H$ for all $t\in U$. Finally, a proof by contradiction establishes that any fiber $X_{t}\cong H$ for $t\in\Delta$.

\begin{exam}\label{ex}
There exists a smooth family of genus $g\ge2$ curves which contains countably (but not finitely) many isomorphic fibers. Denote by $\mathcal{M}_g$ the moduli stack of smooth curves of $g\ge2$.  For a point $p\in \mathcal{M}_g$, let $\Delta_\epsilon$ be an embedded small disk in the smooth affine chart of $p$. The universal family $\mathcal{C}_g$ over $\mathcal{M}_g$ induces a smooth family $\pi_\epsilon:\mathcal{X}_\epsilon\to \Delta_\epsilon$. As $\mathcal{M}_g$ is a Delinge--Mumford stack, the isomorphic fibers of $\pi_\epsilon$ are at most finite. Let $\{a_k\}_{k=1}^\infty$ be a sequence of real numbers satisfies that 
$$0\le a_k<1,\quad \quad \lim_{k\rightarrow +\infty} a_k=1,\quad \quad \sum_{k=1}^{\infty}(1-a_k)<\infty.$$
By \cite[15.21 Theorem]{Ru87}, the Blaschke product 
$$
\phi(z):=\epsilon\prod_{k=1}^{\infty}\frac{a_k-z}{1-a_kz}
$$
is a holomorphic map from the unit disk $\Delta$ to $\Delta_\epsilon$ with infinitely many zero points $\{a_k\}$. Denote by $\mathcal{X}$ the smooth family which is the pullback of $\mathcal{X}_\epsilon$ by $\phi$. Then it is a smooth family of genus $g\ge2$ curves with discrete (but not finitely many) isomorphic fibers. 
\end{exam}

Contrast to Example \ref{ex}, a direct application of Theorem \ref{main theorem} to the smooth family with the compact base yields: 
\begin{cor}\label{1.3}
Let $\pi: \cX\rightarrow Y$ be a smooth family of compact complex manifolds over a compact Riemann surface $Y$, and $H$ a compact hyperbolic complex manifold. Then the $H$-locus 
$$\mathcal{H}:=\{t\in Y: X_t\cong H\}$$
is either at most finite or the whole $Y$.
\end{cor}

In \cite[Corollary 1.3]{lrw}, the first two authors--K. Wang prove that if some fiber of a smooth K\"ahler family is a minimal manifold of general type, then so are all fibers of this family. 
By use of this, E. Viehweg–K. Zuo’s birational isotriviality \cite[Theorem 0.1]{vz} and also the first author–X. Liu’s isotriviality \cite[Theorem 7.1]{LL24}, the first two authors--Wang  
\cite[Corollary 1.6, Theorem 1.7]{lrw} prove that for a smooth K\"ahler family over $\mathbb{P}^1$ or an elliptic curve, if $H$ is a compact projective manifold with the big and nef (or more generally semiample) canonical line bundle, the set $\mathcal{H}$ is either empty or the whole $Y$, while a conjecture of Kobayashi (1970) \cite[(7.4.13) Remark]{Ko98} asserts that the canonical line bundle of a compact (K\"ahler) hyperbolic manifold is ample. Alternatively, the celebrated Lang conjecture \cite{Ln86} predicts that for a smooth projective complex manifold $X$, it is
Kobayashi hyperbolic if and only if $X$ as well as all of its subvarieties are of general
type, while Mori’s breakthrough \cite{Mo79} shows that the hyperbolicity of (at least) the projective $X$ implies the nefness of its canonical bundle, or \cite[Exercise 8 on p. 219]{Db01} shows that the canonical line bundle of a smooth projective
variety of general type and containing no rational curves is actually ample, both due to
the absence of rational curves.

Thus, it is natural to propose:
\begin{conj}
With the setting of Corollary \ref{1.3}, if $Y$ is isomorphic to $\mathbb{P}^1$ or an elliptic curve, then  $\mathcal{H}$ is empty or the whole $Y$.  
\end{conj}

\textbf{Acknowledgements:}  The authors would like to express their gratitude to Professor Kang Zuo for pointing out Footnote \ref{fn1}. We are also sincerely grateful to Professor Ya Deng for suggesting that we consider the pseudo-Brody hyperbolicity analogue of Theorem \ref{main theorem}.

\section{Preliminaries: hyperbolic complex structures and their moduli}
In this section, we will introduce basics on hyperbolic complex structures and their coarse moduli space, to be used to prove Theorem \ref{main theorem}. 
\subsection{Hyperbolicity}
In this subsection, we recall two equivalent definitions of hyperbolic complex manifolds via Kobayashi pseudo-distance and its infinitesimal form, as detailed in \cite{Ro71}. 

We begin by defining the Kobayashi pseudo-distance $d_{X}$ on a complex manifold $X$. For any two points $p,q\in X$, define the function $\rho_{X}: X\times X\rightarrow[0,+\infty]$ by
$$\rho_{X}(p,q)={\inf}_{R}\left\{\frac{1}{2}\ \text{log}\ \frac{R+1}{R-1}:\text{there exists}\ f\in H(\Delta_{R},X)\ \text{with}\ f(0)=p,f(1)=q\right\},$$
where the infimum is taken over all real numbers
$R>1$ and all holomorphic maps $f$ from the open disk $\Delta_{R}=\{z\in \mathbb{C}:|z|<R\}$ to $X$. If no such holomorphic map exists, we set $\rho_{X}(p,q)=+\infty$. 

The \emph{Kobayashi pseudo-distance} $d_{X}$ is then defined as
$$d_{X}(p,q)={\inf}\ \{\Sigma_{i=0}^{k-1}\ \rho_{X}(p_{i},p_{i+1})\},$$
where the infimum is taken over all finite sequences of points $p=p_{0},p_{1},..., p_{k}=q$ in $X$ and all integers $k\geq 1$. Note that $\rho_{X}$ does not necessarily satisfy the triangle inequality, and $d_{X}(p,q)\leq\rho_{X}(p,q)$ always holds. The function $\rho_{X}$ is called the \emph{unreduced Kobayashi distance}. 
\begin{defn}(e.g.,\ \cite[p. 133]{Ro71})
A complex manifold $X$ is said to be \emph{hyperbolic} if the Kobayashi pseudo-distance $d_{X}$ is a distance.
\end{defn}

Given two complex manifolds $X,Y$ and a holomorphic map $g:X\rightarrow Y$, the definition of the Kobayashi pseudo-distance yields
$$d_{Y}(g(p),g(q))\leq d_{X}(p,q)$$
for any $p,\ q\in X$. 
Consequently, the Kobayashi pseudo-distance decreases under the holomorphic map $g$. In particular, if $W\subseteq X$ is a submanifold (with the induced complex structure), then for any $\ p,q\in W,$
$$d_{X}(p,q)\leq d_{W}(p,q).$$

Royden \cite[\S\ 2]{Ro71} defined the \emph{infinitesimal form} $F_{X}$ on a complex manifold $X$: For a point $x\in X$ and a tangent vector $v\in T_{x}X$, 
$$F_{X}(\langle x,v\rangle):={\inf}_{R}\Big\{\frac{1}{R}:R>0,\ \text{there exists}\  \varphi\in H(\Delta_{R},X)\ \text{with}\ \varphi(0)=x,\ \varphi'(0)=v\Big\},$$
where $\varphi'(z_{0})=\varphi_{*}(\frac{\partial}{\partial z})_{z=z_0}$.  
And \cite[Proposition 3]{Ro71} proved that it is upper semi-continuous on the tangent bundle $TX$. 
\begin{defn}(e.g.,\ \cite[p. 133]{Ro71})\label{2}
A complex manifold $X$ is \emph{hyperbolic} if for each point $x\in X$, there exists a coordinate neighborhood $U_{x}$ of $x$ and a constant $c_{x}>0$ such that
\begin{equation*}
F_{X}(\langle y,\eta\rangle)\geq c_{x}\|\eta\|
\end{equation*}
holds for all $\langle y,\eta\rangle\in TU_{x}$, where $TU_{x}$ is the tangent bundle of $U_{x}$ and $\|\cdot\|$ denotes the Euclidean norm. 
\end{defn}

Royden \cite[Theorem 1]{Ro71} established that the Kobayashi pseudo-distance $d_{X}$ is the integrated form of the infinitesimal form $F_{X}$, i.e.,
$$d_{X}(p,q)={\inf}_{\gamma}\int _{\gamma} F_{X}ds,$$
where the infimum is taken over all piecewise $C^{1}$ curves $\gamma$ joining  $p$ to $q$ in $X$.

There are also many other characterizations of hyperbolic complex manifolds in \cite[THEOREM 2]{Ro71} and \cite[THEOREM 2]{HK84}. Recall several equivalent conditions for a compact complex manifold $X$ to be hyperbolic:
\begin{enumerate}
\item A compact complex manifold $X$ is hyperbolic in the sense of Definition \ref{2}.
\item The compact complex manifold $X$ is \emph{tight} with respect to the Kobayashi pseudo-distance $d_X$, that is, the family of holomorphic maps $H(\Delta,X)$ is equicontinuous (with respect to the Euclidean metric in $\Delta$ and the Kobayashi pseudo-distance on $X$).
\item Every entire holomorphic curve in $X$ is constant, i.e., every holomorphic map $f:\mathbb{C}\rightarrow X$  is constant. This is Brody's theorem (cf. \cite[THEOREM 4.1]{Br78}).
\end{enumerate}

\subsection{The coarse moduli space of compact hyperbolic complex manifolds} \label{2.2}
Before introducing this coarse moduli space, we require the notion of a complex space in the complex analytic sense.
\begin{defn}(e.g.,\ \cite[\S\ 2, p. 27]{Re94})
A $\mathbb{C}$-ringed space $(X, \mathcal{O}_X)$ is called a \emph{complex space} if $X$ is Hausdorff and locally isomorphic to a complex model space, that is, for each point $x \in X$, there exist a neighborhood $U \ni x$ and an isomorphism of ringed spaces $(U, \mathcal{O}_X|_U) \simeq (V, \mathcal{O}_V),$  where $(V, \mathcal{O}_V)$ is a given local model.
\end{defn}

Denote by \textbf{An} the category of complex spaces.
For a complex space $B$, let $\mathcal{P}(B)$ be the set of smooth family of compact hyperbolic complex manifold as follows:
\begin{equation*}
\begin{aligned}
\mathcal{P}(B)=&\Big\{(f:\mathcal{X}\rightarrow B): f\ \text{is a smooth family};\\
& f^{-1}(b)\ \text{is a compact hyperbolic complex manifold for any}\ b\in B\Big\}.
\end{aligned}
\end{equation*}

The arrows in $\mathcal{P}(B)$ are defined as follows:
$(f_1:\mathcal{X}_{1}\rightarrow B)$ to $(f_2:\mathcal{X}_{2}\rightarrow B)$ are  $B$-isomorphisms \begin{equation*}
\CD
  \Phi:\mathcal{X}_{1} @> \cong>> \mathcal{X}_{2}. \\
\endCD
\end{equation*}
Therefore, $\mathcal{P}$ forms a groupoid fibered over the category of complex spaces \textbf{An}. Let $p:\mathcal{P}\to \textbf{An}$ be the natural functor. 
A \emph{coarse moduli space} for $\mathcal{P}$ is a complex space $\mathfrak{M}$ equipped with a natural transformation
$$\alpha: \mathcal{P} \to \operatorname{Hom}_{\textbf{An}}(-, \mathfrak{M})$$
satisfying:  

\noindent($\star$) Bijectivity at points: For the one-point space $\text{\{pt\}}$, the map $$\alpha_{\{\text{pt}\}}: \mathcal{P}(\{\text{pt}\})\longrightarrow\operatorname{Hom}_{\textbf{An}}(\{\text{pt}\},\mathfrak{M})$$ is bijective (canonically identifying $\operatorname{Hom}_{\textbf{An}}(\{\text{pt}\},\mathfrak{M})$ with the underlying set of $\mathfrak{M}$);

\noindent($\star\star$) Universal property: For any complex space $\mathfrak{N}$ and any natural transformation $$\eta:\mathcal{P}\longrightarrow\operatorname{Hom}_{\textbf{An}}(-,\mathfrak{N}),$$
there exists a unique natural transformation
$$\beta: \operatorname{Hom}_{\textbf{An}}(-,\mathfrak{M})\longrightarrow\operatorname{Hom}_{\textbf{An}}(-,\mathfrak{N})$$
such that
$$\beta\circ \alpha=\eta.$$

According to \cite[THEOREM 5]{Sc91}, $\mathcal{P}$ admits a coarse moduli space if it satisfies the following:
\begin{enumerate}[a)]
\item Existence of universal deformations: For any object $a$ with $p(a)=\{\text{pt}\}$, there exists a universal deformation of $a$;
\item Representability and properness of $\text{\bf{Isom}}$ functors: For any complex space $B$ and objects $c,d\in \text{Ob}\mathcal{P}$ with $p(c)=p(d)=B$, the functor $$\text{\bf{Isom}}_{B}(c,d):\textbf{An}/B\rightarrow\textbf{Sets}$$
     is representable by a proper morphism $\tau:T\rightarrow B$, where for each $b\in B$, the fiber $T_{b}$ is canonically identified with the set of isomorphisms $\text{Isom}(c_{b},d_{b})$ as a topological space.
\end{enumerate}

Recall that a deformation $f:(\mathcal{X},X_0)\rightarrow (B,b_{0})$ of a compact complex manifold $X_0$ is \emph{versal} if, for any deformation $f':(\mathcal{X'},X_0)\rightarrow (D,d_{0})$ of $X_0$, there exist a neighborhood $U_{d_0}\subset D$ of $d_0$ and a holomorphic map $\alpha:U_{d_0}\rightarrow B$ with $\alpha(d_{0})=b_0$ such that the restricted family $\mathcal{X'}|_{U_{d_0}}$ is isomorphic to the pullback family $\mathcal{X}\times_{B}U_{d_0}$. The deformation is \emph{universal} if, in addition, the map $\alpha$ is unique.

Let $f:\mathcal{X}\rightarrow B$ and $g:\mathcal{Y}\rightarrow B$ be smooth families of compact hyperbolic complex manifolds over a complex space $B$. Define the functor
$$\text{\bf{Isom}}_{B}(\mathcal{X},\mathcal{Y}):\textbf{An}/B\rightarrow \textbf{Sets}$$
by assigning to each complex space  $\mathcal{T}$ over $B$ the set
$$\text{\bf{Isom}}_{B}(\mathcal{X},\mathcal{Y})(\mathcal{T})= \{ \text{the isomorphisms}\ \varphi: \mathcal{X}\times_{B}\mathcal{T}\rightarrow \mathcal{Y}\times_{B}\mathcal{T}\}.$$

We now verify that the fibered groupoid $\mathcal{P}$ satisfies both conditions a) and b). 
By \cite[(5.4.4) Theorem]{Ko98}, the holomorphic automorphism group $\text{Aut}(X)$ of any compact hyperbolic complex manifold
$X$ is finite. Consequently, \cite[COROLLARY of THEOREM 4.1]{Wa69} implies that $X$ admits a universal deformation. This verifies condition a).

According to \cite[Part 3]{Kh07}, the functor $\text{\bf{Isom}}_{B}(\mathcal{X},\mathcal{Y})$ is representable by a proper morphism $$\text{Isom}_{B}(\mathcal{X},\mathcal{Y})\rightarrow B,$$ 
where  $\text{Isom}_{B}(\mathcal{X},\mathcal{Y})$ denotes  the set of isomorphisms from $\mathcal{X}$ to $\mathcal{Y}$ over $B$. This establishes condition b).

Thus, the groupoid $\mathcal{P}$ admits a coarse moduli space $\mathfrak{M}$ whose points parameterize the isomorphism classes of compact hyperbolic complex manifolds. Moreover, a smooth family $f:\mathcal{X}\rightarrow B$ of such manifolds induces a holomorphic morphism 
$$\tilde{f}:B\rightarrow\mathfrak{M}$$
sending each point $b\in B$ to the isomorphism class of hyperbolic structure of the fiber $X_{b}$.

Although the existence of a coarse moduli space $\mathfrak{M}$ for $\mathcal{P}$ is established in this section and in \cite[THEOREM 7]{Wr77}, respectively, the detailed complex structure of $\mathfrak{M}$ is not explicitly constructed in these references. We establish this structure in the next section.

\section{Deformation rigidity of hyperbolic complex structures}
In this section, we provide the proof of Theorem \ref{main theorem}. We begin by explicitly characterizing the complex structure of $\mathfrak{M}$. We need: 
\begin{lemma}[{\cite{Br78}}]\label{open'}
Let $M$ be a fixed compact smooth manifold with a Hermitian metric. Suppose $M$ admits a family of complex structures smoothly parameterized by a complex space $S$. For each $s\in S$, denote by $M_s$ the corresponding complex manifold. Then the set
$$\{s\in S: M_{s}\ \text{is hyperbolic}\}$$
is open in $S$.
\end{lemma}
\begin{proof} \cite[THEOREM 3.1]{Br78} shows the function $D:S\rightarrow (0,\infty]$ defined by
$$D(s):=\mathrm{sup}_{f\in \mathrm{Hol}(\Delta,M_s)}|f'(0)|$$
is continuous, while \cite[LEMMA 1.1]{Br78} tells us that $D(s)<+\infty$ if and only if $M_s$ is hyperbolic. \end{proof}

\begin{prop}[{\cite[THEOREM 7]{Wr77}}]\label{sep}
Let $\mathcal{P}$ the groupoid fibered over $\mathbf{An}$ as in Section \ref{2.2}.
Then the set $\mathfrak{M}$ of such isomorphism classes admits a natural complex space structure and is the coarse moduli space for $\mathcal{P}$.  
\end{prop}

\begin{proof} The proof is divided into four steps: 
\setcounter{step}{0}
\renewcommand{\thestep}{(\Roman{step})} 
\begin{step}\label{step1}
Local complex structure on $\mathfrak{M}$.
\end{step}
Let  $X$ be a compact complex manifold and $I_{0}$ be a hyperbolic complex structure on $X$. Let $f:(\mathcal{X},X)\rightarrow (S,s_0)$ be a universal deformation of $X$ with the structure $I_0$. Denote by $\text{Aut}(I_0)$ the automorphism group of the structure $I_0$.

Since $f:(\mathcal{X},X)\rightarrow (S,s_0)$ is universal, for any holomorphic family $f':(\mathcal{X}',X)\rightarrow (T,t_0)$ of $X$ with the structure $I_0$, there exists a neighborhood $T'$ of $t_0$ and a unique holomorphic map $\alpha:T'\rightarrow S$ with $\alpha(t_0)=s_0$, and an isomorphism $\Phi:\mathcal{X}'|_{T'}\rightarrow \mathcal{X}\times_{S}T'$ over $T'$. This induces a commutative diagram:
\begin{equation*}\label{nbhd iso}
\CD
  \mathcal{X}'|_{T'} @> \Phi>> \mathcal{X} \\
  @V f'|_{T'} VV @V f VV  \\
  T' @>\alpha>> S,
\endCD
\end{equation*}
where, for each $t\in T'$, the restriction $\Phi|_{X'_{t}}:X'_{t}:=f'^{-1}(t)\rightarrow X_{\alpha_{(t)}}:=f^{-1}(\alpha(t))$ is an isomorphism. 

Now consider the case $\mathcal{X'}=\mathcal{X}$. Since $I_0$ is compact and hyperbolic, $\text{Aut}(I_0)$ is finite by {\cite[(5.4.4) Theorem]{Ko98}}. Therefore, there exists a neighborhood $U$ of $s_{0}$ in $S$ such that $\text{Aut}(I_0)$ acts holomorphically on $\mathcal{X}|_{U}$ and thus on $U$. By {\cite[THEOREM 6]{Wr77}}, one can choose a sufficiently small neighborhood $W'$ of $s_0$ such that the set $$A=\bigcup_{s',s''\in \overline{W'}}\text{Isom}(X_{s'},X_{s''})$$ is compact. 
Following {\cite[Lemma 4.1]{NS68}}, one may shrink $U$ so that for any $s',s''\in U$, a morphism $\phi:X_{s'}\rightarrow X_{s''}$ is an isomorphism if and only if $\phi\in\text{Aut}(I_0)$ and $\phi(s')=s''$. Equivalently, such isomorphism exists if and only if $s'$ and $s''$ lie in the same $\text{Aut}(I_0)$-orbit. 
Furthermore, by Lemma \ref{open'}, one can assume the family $f:\mathcal{X}\rightarrow (S,s_0)$ satisfies that:
\begin{enumerate}[i)]
\item Each fiber $X_{s}$ is hyperbolic for $s\in S$;
\item $\text{Aut}(I_0)$ acts on $S$ and, for $s',\ s''\in U,\  \phi:X_{s'}\rightarrow X_{s''}$ is isomorphic if and only if $s',\ s''$ are in the same $\text{Aut}(I_0)$-orbit.
\end{enumerate}

Since $\text{Aut}(I_0)$ is finite and acts holomorphically on $S$, the quotient $S/\text{Aut}(I_0)$ has a natural structure of a complex space from Cartan's theorem \cite[Th\'eor\`eme 4]{Ca57}. 
The canonical map $S\rightarrow\mathfrak{M}$, sending $s\in S$ to the hyperbolic isomorphism class of $X_{s}$, factors through the quotient and induces an injective holomorphic map $$S/\text{Aut}(I_0)\rightarrow\mathfrak{M}.$$ 
This map provides a ``local coordinate chart'' for $\mathfrak{M}$ at the isomorphism class of $I_0$.

\begin{step}\label{step2}Glue local charts to define the complex space structure on $\mathfrak{M}$.
\end{step}
We now show that the local coordinate charts constructed in Step \ref{step1} can be glued together to endow
$\mathfrak{M}$ with a complex space structure. 

Let $I_{1}$ and $I_2$ be two hyperbolic complex structures on $X$. For each $i=1,\ 2$, let $\mathcal{X}_{i}\rightarrow S_{i}$ be a family satisfying conditions i) and ii) from Step \ref{step1} with the structure $I_i$. 
Denote by 
$$p_{i}:S_i\rightarrow S_{i}':=S_{i}/\text{Aut}(I_i)$$ 
the canonical quotient map, which is a morphism of complex spaces. Let $\gamma_{i}:S_{i}'\rightarrow\mathfrak{M}$ be the corresponding canonical injective holomorphic map. 
We need to prove that if the intersection $V:=\gamma_{1}(S_{1}')\cap \gamma_{2}(S_{2}')$ is non-empty, then:
\begin{enumerate}[(a)]
\item $\gamma_{i}^{-1}(V)$ is open in $S_{i}',$ for each $i=1,2$;
\item The map $\gamma_{2}^{-1}\circ \gamma_{1}:\gamma_{1}^{-1}(V)\rightarrow
    \gamma_{2}^{-1}(V)$ is an isomorphism of complex spaces.
\end{enumerate}

Suppose $V\neq\emptyset$. Then there exist points $s_1\in S_1$ and $s_2\in S_2$ such that the structures $I_{s_1}$(on $X_{s_1}$) and $I_{s_2}$(on $X_{s_2}$) are isomorphic. Let $\varphi:X_{s_1}\rightarrow X_{s_2}$ be such an isomorphism. By the universality property (as encoded in the commutative diagram from Step \ref{step1}), $\varphi$ extends to an isomorphism of  families over an isomorphism $\widetilde{\varphi}:U_{1}\rightarrow U_2$ of suitable neighborhoods $U_{1}\subseteq S_1$ of $s_1$ and $U_{2}\subseteq S_2$ of $s_2$: 
\begin{equation*}
\CD
  \Phi:\mathcal{X}_{1}|_{U_{1}} @> \cong>> \mathcal{X}_{2}|_{U_{2}}. \\
\endCD
\end{equation*}
Since $\text{Aut}(I_i)$ is finite and $\text{Aut}(X_{s_{i}})$ is precisely the isotropy subgroup of $\text{Aut}(I_i)$ at $s_{i}$, for each $i=1,2$, one can choose the neighborhood $U_{i}$ of $s_{i}$ such that:
\begin{enumerate}[1)]
\item $U_{i}$ is $\text{Aut}(X_{s_{i}})$-stable;
\item Two points in $U_{i}$ are $\text{Aut}(I_i)$-equivalent if and only if they are $\text{Aut}(X_{s_{i}})$-equivalent.
\end{enumerate}
As the quotient map $p_{i}$ is open, $p_{i}(U_{i})$ is open in $S_{i}'$, establishing condition (a).
Moreover, the quotient map $p_i$ restricts to a canonical isomorphism $p_{i}(U_{i})\cong U_{i}/\text{Aut}(X_{s_{i}})$. The map
$$\gamma_{2}^{-1}\circ \gamma_{1}|_{p_{1}(U_{1})}:p_{1}(U_{1})\rightarrow p_{2}(U_{2})$$
is induced by the isomorphism $\widetilde{\varphi}:U_{1}\rightarrow U_{2}$ underlying the family isomorphism $\Phi$. Since $\widetilde{\varphi}$ is a biholomorphism and intertwines the group actions (by the choice of $U_i$ and property 2)), it descends to an isomorphism
\begin{equation*}
\CD
  U_{1}/\text{Aut}(X_{s_1}) @> \cong>> U_{2}/\text{Aut}(X_{s_2}). \\
\endCD
\end{equation*}
Consequently, the composition $\gamma_{2}^{-1}\circ \gamma_{1}:\gamma_{1}^{-1}(V)\rightarrow\gamma_{2}^{-1}(V)$ is an isomorphism of complex spaces, proving condition (b).
\begin{step} \label{step3} Establish the Hausdorff property for $\mathfrak{M}$.
\end{step}
We now show that $\mathfrak{M}$ carries the Hausdorff topology. Let $I_{1}$ and $I_{2}$ be distinct hyperbolic complex structures on $X$. 
\begin{claim}\label{claim}
There exist universal deformations $\mathcal{X}\rightarrow (S,s_0)$ with $I_{1}$ and $\mathcal{Y}\rightarrow (T,t_0)$ with $I_2$ satisfying
conditions i) and ii) in Step \ref{step1}, such that for all $s\in S$ and $t\in T$, the fibers $X_{s}$ and $Y_{t}$ are not isomorphic.
\end{claim}
Assume that Claim \ref{claim} holds, and then the images of $S$ and $T$ under the canonical maps $S\rightarrow \mathfrak{M}$ and $T\rightarrow \mathfrak{M}$ are disjoint neighborhoods in $\mathfrak{M}$ containing the isomorphism classes of $I_1$ and $I_2$, respectively. This confirms $\mathfrak{M}$ is Hausdorff. 

We now prove Claim \ref{claim}. Since $I_{1}$ and $I_{2}$ are hyperbolic, universal deformations $\mathcal{X}\rightarrow (S,s_{0})$ with $I_1$ and $\mathcal{Y}\rightarrow (T,t_0)$ with $I_2$ exist and satisfy conditions i) and ii) in Step \ref{step1}. By \cite[(3.11.9) Corollary]{Ko98}, the infinitesimal Kobayashi metrics $F_{X_{s}}$ and $F_{Y_{t}}$ are continuous on $\bigcup_{s\in S}TX_{s}$ and $\bigcup_{t\in T}TY_{t}$, respectively. Thus, \cite [THEOREM 5]{Wr77} implies that the Kobayashi distances $d_{X_{s}}$ and $d_{Y_{t}}$ vary continuously. Consequently, there exist neighborhoods $U_{s_0}\subseteq S$ of $s_0$, $U_{t_0}\subseteq T$ of $t_0$, and a constant $C>0$ such that for all $s\in U_{s_0},\ t\in U_{t_0}$, and all points $x_{1},\ x_{2}\in X_{s_0},\ y_{1},\ y_{2}\in Y_{t_0},$ the inequalities hold:
\begin{equation*}\label{equ}
\begin{aligned}
\frac{1}{C}d_{X_s}(x_{1},x_{2})\leq d_{X_{s_0}}(x_{1},x_{2})\leq Cd_{X_s}(x_{1},x_{2}),\\
\frac{1}{C}d_{Y_{t}}(y_{1},y_{2})\leq d_{Y_{t_0}}(y_{1},y_{2})\leq Cd_{Y_t}(y_{1},y_{2}).
\end{aligned}
\end{equation*}

We argue by contradiction. Suppose that some $X_s$ and $Y_t$ are isomorphic. Then there exist sequences $s_n\rightarrow s_0$ in $U_{s_0}$, $t_n\rightarrow t_0$ in $U_{t_0}$, and holomorphic isomorphisms $\varphi_{n}:X_{s_n}\rightarrow Y_{t_n}$. For all $n$ and all $x_{1},\ x_{2}\in X_{s_{n}}$, we have
$$d_{Y_{t_n}}(\varphi_{n}(x_{1}),\varphi_{n}(x_{2}))\leq Cd_{Y_{t_0}}(\varphi_{n}(x_{1}),\varphi_{n}(x_{2}))\leq Cd_{X_{s_0}}(x_{1},x_{2})\leq C^{2}d_{X_{s_n}}(x_{1},x_{2}),$$
where the second inequality holds because the Kobayashi distance decreases under holomorphic maps.
Thus, the sequence $\{\varphi_{n}:(X_{s_n},d_{X_{s_n}})\rightarrow (Y_{t_{n}},d_{Y_{t_{n}}})\}$ is equicontinuous. Since $X$ is compact, by the Ascoli--Arzel\`a theorem, $\{\varphi_{n}\}$ has a subsequence $\{\varphi_{n_{k}}\}$ that converges uniformly to a homeomorphism $\varphi:X_{s_{0}}\rightarrow Y_{t_{0}}$. Since both $\mathcal{X}$ and $\mathcal{Y}$ are locally biholomorphic to products, $\varphi$ is holomorphic according to Montel's theorem, contradicting the assumption that $I_{1}$ and $I_2$ are nonisomorphic. Therefore, no such isomorphisms $\varphi_{n}$ exist. After shrinking $S$ to $U_{s_0}$ and $T$ to $U_{t_0}$, the families satisfy Claim \ref{claim}.

\begin{step}
The complex space $\mathfrak{M}$ is a coarse moduli space for the fibered category $\mathcal{P}$.
\end{step}
To verify this, we show that $\mathfrak{M}$ satisfies:

\noindent($\star$) Bijectivity on points:
$$\mathcal{P}(\text{\{pt\}})=\{\text{isomorphism classes of compact hyperbolic complex manifolds}\}$$
corresponds bijectively to the set of closed points $\mathfrak{M}(\text{\{pt\}})$ of $\mathfrak{M}$,  yielding
$$\mathcal{P}(\operatorname{\{pt\}})\cong \operatorname{Hom}_{\textbf{An}}(\operatorname{\{pt\}},\mathfrak{M})=\mathfrak{M}(\operatorname{\{pt\}}).$$

\noindent($\star\star$) Universal property: From the proof of {\cite[Theorem 2]{Fu84}}, universal deformations induce a natural transformation
$$\alpha: \mathcal{P}\longrightarrow\operatorname{Hom}_{\textbf{An}}(-,\mathfrak{M}).$$
Moreover, for any $\mathfrak{N}\in \text{Ob}{\textbf{An}}$ and any natural transformation
$$\eta:\mathcal{P}\longrightarrow\operatorname{Hom}_{\textbf{An}}(-,\mathfrak{N}),$$
there exists a unique natural transformation
$$\beta: \operatorname{Hom}_{\textbf{An}}(-,\mathfrak{M})\longrightarrow\operatorname{Hom}_{\textbf{An}}(-,\mathfrak{N})$$
such that
$$\beta\circ \alpha=\eta.$$ 
Thus, $\mathfrak{M}$ is a complex space and the coarse moduli space for $\mathcal{P}$.
\end{proof}

\begin{cor}
\label{analytic}
For a smooth family of compact hyperbolic complex manifolds over an open disk $\Delta$, denote by $E\subseteq\Delta$ the set of points whose fibers are biholomorphic to a fixed compact hyperbolic manifold $H$. Then $E$ is either at most a discrete subset of $\Delta$ or the whole $\Delta$.
\end{cor}
\begin{proof}
Proposition \ref{sep} shows that for the coarse moduli space $\mathfrak{M}$ of compact hyperbolic complex manifolds, there exists a holomorphic morphism $\sigma:\Delta\rightarrow\mathfrak{M}$ sending each point $t\in\Delta$ to the corresponding isomorphism class $[X_t]$ of the hyperbolic manifold $X_{t}$ in $\mathfrak{M}$. The image $\sigma(E)$ consists of a single point $[H]$. By the Hausdorff property of $\mathfrak{M}$ according to Proposition \ref{sep}, the point $[H]$ is a closed point in $\mathfrak{M}$. Since $\sigma:\Delta\rightarrow\mathfrak{M}$ is holomorphic, if its preimage $\sigma^{-1}([H])$ is a subset of $\Delta$ with accumulation points in $\Delta$,  it follows that $\sigma^{-1}([H])$ must be the whole set $\Delta$. Thus, $X_{t}\cong H$ for all $t\in \Delta$.
\end{proof}

\begin{lemma}[{\cite[COROLLARY 4.1]{Wr77}}]\label{limit}
Let $\pi:\mathcal{X}\rightarrow\Delta$ be a deformation of a compact complex manifold $X_0$. 
Suppose that there exists a subset $E\subseteq\Delta$ with $0$ as a limit point such that $X_{t}\cong H$ for all $t\in E$, where $H$ is hyperbolic. Then $X_0$ is hyperbolic.
\end{lemma}

We now prove Theorem  \ref{main theorem},  restated in an equivalent form:
\begin{theo}
Let $\pi: \cX\rightarrow\Delta$ be a smooth family of compact complex manifolds and $H$ a compact hyperbolic complex manifold. Assume that there is a subset $E\subseteq\Delta$ with accumulation points in $\Delta$ such that $X_{t}\cong H$ for any $t\in E$. Then all fibers $X_{t}\cong H$ for $t\in\Delta$.
\end{theo}

\begin{proof}
Choose an accumulation point $t_0$ in the assumption. Then Lemma \ref{open'} gives rise to a small open disk $\Delta_{t_{0}}\subset\Delta$ centered at $t_0$, such that all fibers $X_t$ are hyperbolic for $t\in\Delta_{t_{0}}$. 
By Corollary \ref{analytic}, $X_t\cong H$ for all $t\in \Delta_{t_{0}}$.

Fix a point $q\in \Delta_{t_{0}}$, and let $p$ be an arbitrary point in $\Delta\setminus \Delta_{t_{0}}$. Let $$\gamma:[0,1]\rightarrow\Delta$$ be a line segment connecting $\gamma(0)=q$ and $\gamma(1)=p$. Now set
$$h:=\text{sup}\{u\in [0,1]:X_{\gamma(u)}\cong H\}.$$
Obviously, $h>0$ and actually $h=1$. In fact, if $h<1$, then Lemmata \ref{limit} and \ref{open'} give rise to a neighborhood $U_{\gamma(h)}\subseteq\Delta$ of $\gamma(h)$ such that all fibers over $U_{\gamma(h)}$ are hyperbolic and $U_{\gamma(h)}$ contains a line segment in $\gamma([0,h])$. 
By Corollary \ref{analytic}, the fibers $X_{t}\cong H$ for all $t\in{U_{\gamma(h)}}$. This contradicts the supremum property of $h$. This proves $h=1$.
By the same reasoning, one can also prove $X_{\gamma(1)}\cong H$, i.e., $X_{p}\cong H$.

In summary, we prove that all fibers over $\Delta$ are biholomorphic to $H$.
\end{proof}

\end{document}